\documentclass{amsart}
\usepackage{amscd,amsmath,amssymb,amsthm,amsfonts,epsfig,graphics}

\theoremstyle{theorem}
\newtheorem{theorem}{Theorem}
\theoremstyle{definition}
\newtheorem{definition}{Definition}

\newtheorem{remark}{Remark}

\title{Submanifolds in metallic Riemannian manifolds}
\author{Cristina E. Hretcanu and Adara M. Blaga}


     \numberwithin{equation}{section}

\begin{document}

\maketitle

\begin{abstract}
The aim of our paper is to focus on some properties of submanifolds in Riemannian manifolds endowed with
endomorphisms that generalize the Golden Riemannian structure, named metallic Riemannian structures.
We focus on the properties of the structure induced on submanifolds, named by us $\Sigma $-metallic Riemannian structures, especialy regarding the normality of this types of structure. Examples of structures induced on a sphere of codimension 1 by some metallic Riemannian structures defined on an Euclidean space are given.

\end{abstract}

\section{Introduction}
\label{intro}
{\it Metallic means family} was introduced by Vera W. de Spinadel in (\cite{Spinadel}) and it contains some generalizations of the Golden mean, such as {\it the Silver mean, the Bronze mean, the Copper mean, the Nickel mean} and many others.

A $(p, q)$-{\it metallic number} is the positive solution of the equation $x^{2}-px-q=0$ (for fixed positive integer values of $p$ and $q$) and it has the form:
\begin{equation}\label{e1}
\sigma_{p,q}=\frac{p+\sqrt{p^{2}+4q}}{2}.
\end{equation}

Some important members of the metallic mean family (\cite{Spinadel}) are the following: the Golden mean $\phi =\frac{1+\sqrt{5}}{2}$ (for $p=q=1$), the Silver mean $\sigma _{Ag}=\sigma _{2, 1}=1+\sqrt{2}$ (for $q=1$ and $p=2$), the Bronze mean $\sigma _{Br}=\sigma _{3, 1}=\frac{3+\sqrt{13}}{2}$ (for $q=1$ and $p=3$), the Subtle mean $\sigma_{4, 1}=2+\sqrt{5}=\phi^{3}$ (for $p=4$ and $q=1$), the Copper mean $\sigma _{Cu}=\sigma _{1, 2}=2$ (for $p=1$ and $q=2$), the Nickel mean $\sigma _{Ni}=\sigma _{1, 3}=\frac{1+\sqrt{13}}{2}$ (for $p=1$ and $q=3$) and so on.

Polynomial structures on manifolds were defined by S. I. Goldberg, K. Yano and N. C. Petridis in (\cite{Goldberg1} and \cite{Goldberg2}). C. E. Hretcanu and M. Crasmareanu  defined some particular cases of  polynomial structures, called \textit{Golden structure} (\cite{Hr2}, \cite{Hr3}, \cite{CrHr}, \cite{CrHrMu}) and some generalizations of this, called \textit{metallic structure} (\cite{Hr4}).

In this paper, after recalling the basic notions of Golden and metallic Riemannian manifolds (in Section 2), we focus on the properties of $\Sigma $-metallic Riemannian structures induced on submanifolds by metallic Riemannian structure (in Section 3) and we give some examples of structures induced on submanifolds by metallic Riemannian structures (in Section 4). In the next section we find some properties of the normal $\Sigma $-metallic Riemannian structures.

\section{Golden and metallic Riemannian structures}
\label{sec:1}
The \textit{Golden structure} (\cite{CrHr}) is a polynomial structure on a manifold $\overline{M}$ determined by a $(1,1)$-tensor field $J$ which satisfies:
 \begin{equation}\label{e2}
 J^{2}= J + I,
 \end{equation}
where $I$ is the identity operator on the Lie algebra $\mathcal{X}(\overline{M})$ of
vector fields on $\overline{M}$.

A {\it metallic structure} (\cite{Hr4}) is a polynomial structure on a manifold $\overline{M}$ determined by a $(1, 1)$-tensor field $J$ which satisfies:
\begin{equation}\label{e3}
J^{2}= pJ+qI,
\end{equation}
where $I$ is the identity operator on the Lie algebra $\mathcal{X}(\overline{M})$ of vector fields on $\overline{M}$ and $p, q$ are fixed positive integer numbers.

Moreover, if $(\overline{M}, \overline{g})$ is a Riemannian manifold endowed with a metallic structure $J$ such that the Riemannian metric $\overline{g}$ is $J$-compatible, i.e.:
\begin{equation} \label{e4}
\overline{g}(JX, Y)= \overline{g}(X, JY),
\end{equation}
for any $X, Y \in \mathcal{X}(\overline{M})$, then $(\overline{g}, J)$ is called {\it metallic Riemannian structure} and $(\overline{M},\overline{g}, J)$ is a {\it metallic Riemannian manifold}.

In a metallic Riemannian manifold $(\overline{M},\overline{g}, J)$, from (\ref{e3}) and (\ref{e4}) we get:
\begin{equation}\label{e5}
\overline{g}(JX, JY)=p \overline{g}(X, JY)+q \overline{g}(X, Y),
\end{equation}
for any $X, Y\in \mathcal{X}(\overline{M})$.

An almost product structure $F$ on a Riemannian manifold $(\overline{M},\overline{g})$ induces two metallic Riemannian structures on $\overline{M}$, given by (\cite{Hr4}):
\begin{equation}\label{e6}
J_{1}=\frac{p}{2}I +\left(\frac{2\sigma _{p, q}-p}{2}\right)F, \quad J_{2}=\frac{p}{2}I-\left(\frac{2\sigma _{p, q}-p}{2}\right)F.
\end{equation}

Conversely, every metallic structure $J$ on a Riemannian manifold $(\overline{M},\overline{g})$ induces two almost product structures on this manifold  (\cite{Hr4}):
\begin{equation}\label{e7}
F_{\pm }=\pm \left(\frac{2}{2\sigma _{p, q}-p}J-\frac{p}{2\sigma _{p, q}-p}I\right).
\end{equation}

\section{$\Sigma$-metallic Riemannian structures}
\label{sec:2}
\normalfont

In this section we assume that $M$ is an $n$-dimensional submanifold isometrically immersed in the $(n+r)$-dimensional metallic Riemannian manifold ($\overline{M}, \overline{g}, J)$, with $n, r \geq 1$ integer numbers.

We denote by $T_{x}(M)$ the tangent space of $M$ in a point $x \in M$ and by $T_{x}^{\bot }(M)$ the normal space of $M$ in $x$. Let $i_{*}$ be the differential of the immersion $i: M \rightarrow\overline{M}$. The induced Riemannian metric $g$ on $M$ is given by $g(X, Y)=\overline{g}(i_{*}X, i_{*}Y)$, for any $X, Y \in \mathcal{X}(M)$.
We fixe a local orthonormal basis $\{N_{1},...,N_{r}\}$ of the normal space $T_{x}^{\bot}(M)$. Hereafter we assume that the indices $\alpha, \beta, \gamma $ run over the range $\{1,..., r\}$.

The vector fields $J(i_{*}X)$ and $J(N_{\alpha})$ can be decomposed into tangential and normal components:
\begin{equation}\label{e100}
(i) \: J(i_{*}X)=i_{*}(PX)+\sum_{\alpha=1}^{r}u_{\alpha}(X)N_{\alpha}, \quad (ii) \: J(N_{\alpha})=i_{*}(\xi _{\alpha})+\sum_{\beta=1}^{r}a_{\alpha \beta} N_{\beta},
\end{equation}
for any $X \in T_{x}(M)$, where $P$ is a $(1, 1)$-tensor field on $M$, $\xi _{\alpha}$ are vector fields on $M$, $u_{\alpha}$ are $1$-forms on $M$ and $(a_{\alpha\beta})_{r}$ is an $r \times r$ matrix of smooth real functions on $M$.

In the rest of the paper we shall simply denote by $X$ the vector field $i_{*}X$, for $X \in \mathcal{X}(M)$.

\begin{theorem} (\cite{Hr4})
The structure $\Sigma =(P, g, u_{\alpha}, \xi _{\alpha},(a_{\alpha\beta})_{r})$, induced on a submanifold $M$ by the metallic Riemannian structure $(\overline{g}, J)$ on $\overline{M}$, satisfies the following equalities:
\begin{equation}\label{e9}
     P^{2}X=pPX+qX-\sum_{\alpha}u_{\alpha}(X)\xi _{\alpha},
\end{equation}
\begin{equation}\label{e10}
  (i) \ \ u_{\alpha}(PX)=pu_{\alpha}(X)-\sum_{\beta}a_{\alpha\beta}u_{\beta}(X), \quad
  (ii)\  a_{\alpha\beta}=a_{\beta \alpha},
  \end{equation}
\begin{equation}\label{e11}
  (i) \ \ u_{\beta}(\xi _{\alpha})=q\delta_{\alpha \beta}+pa_{\alpha\beta}-\sum_{\gamma }a_{\alpha \gamma }a_{\gamma \beta}, \quad
  (ii) \  P\xi _{\alpha}=p\xi _{\alpha}-\sum_{\beta}a_{\alpha\beta}\xi _{\beta},
\end{equation}
\begin{equation} \label{e12}
      (i) \ \ g(PX, Y)=g(X, PY), \quad
      (ii) \ u_{\alpha}(X)=g(X, \xi_{\alpha})
\end{equation}
for any $X, Y\in \mathcal{X}(M)$, where $\delta_{\alpha \beta}$ is the Kronecker delta and $p, q$ are fixed positive integers.
\end{theorem}

\begin{definition}
    A $\Sigma$-metallic Riemannian structure on a Riemannian manifold $(M, g)$ is given by the data $\Sigma =(P, g, u_{\alpha }, \xi _{\alpha }, (a_{\alpha \beta})_{r})$, determined by a $(1, 1)$-tensor field $P$ on $M$, the vector fields $\xi_{\alpha }$ on $M$, the $1$-forms $u_{\alpha }$ on $M$, the $r\times r$ matrix $(a_{\alpha \beta })_{r}$ of smooth real functions on $M$ which verify the relations (\ref{e9}), (\ref{e10}), (\ref{e11}) and (\ref{e12}).
\end{definition}

\begin{remark}
If $\Sigma =(P, g, u_{\alpha}, \xi _{\alpha},(a_{\alpha\beta})_{r})$ is the induced structure on the submanifold $M$ in the metallic
Riemannian manifold $(\overline{M},\overline{g},J)$, then $M$ is an invariant submanifold with respect to $J$ (i.e. $J(M) \subset TM$)
if and only if $(M,g,P)$ is a metallic Riemannian manifold, whenever $P$ is non-trivial (\cite{Hr4}).
\end{remark}

Let $\overline{\nabla}$ and $\nabla$ be the Levi-Civita connections on $(\overline{M},\overline{g})$ and $(M,g)$, respectively. The Gauss and Weingarten formulas are:
\begin{equation}\label{e13}
(i) \: \overline{\nabla}_{X}Y=\nabla_{X}Y+\sum_{\alpha=1}^{r}h_{\alpha}(X,Y)N_{\alpha}, \quad (ii) \:  \overline{\nabla}_{X}N_{\alpha}=-A_{\alpha}X+\nabla_{X}^{\bot}N_{\alpha},
\end{equation}
respectively, where $A_{\alpha}=:A_{N_{\alpha}}$ is the shape operator and $h(X,Y)=\sum_{\alpha=1}^{r}h_{\alpha}(X,Y)N_{\alpha}$ is the second fundamental form.
Also,
\begin{equation}\label{e130}
 h_{\alpha}(X, Y)=g(h(X,Y),N_{\alpha})=g(A_{\alpha}X, Y),
\end{equation}
for any $X, Y \in \mathcal{X}(M)$.

\begin{remark}
The normal connection $\nabla_{X}^{\bot}N_{\alpha}$ has the decomposition:
\begin{equation}\label{e14}
\nabla_{X}^{\bot}N_{\alpha}=\sum_{\beta=1}^{r}l_{\alpha\beta}(X)N_{\beta},
 \end{equation}
for any $X \in \mathcal{X}(M)$, $\alpha \in \{1,..., r\}$, and
\begin{equation}\label{e140}
l_{\alpha\beta}=-l_{\beta\alpha},
\end{equation}
 for any $\alpha, \beta \in \{1,..., r\}$.
\end{remark}
\begin{proof}
From $\overline{g}(\nabla_{X}^{\bot}N_{\alpha}, N_{\beta})+\overline{g}(N_{\alpha}, \nabla_{X}^{\bot}N_{\beta})=0$ we get:
$$\overline{g}(\sum_{\gamma }l_{\alpha\gamma }(X)N_{\gamma}, N_{\beta})+\overline{g}(N_{\alpha}, \sum_{\gamma }l_{\beta\gamma }(X)N_{\gamma })=0,$$
for any $X\in \mathcal{X}(M)$ and $\gamma \in \{1,..., r\}$. Thus we obtain (\ref{e140}).
\end{proof}

Using an analogy of a locally product manifold (\cite{Pitis}), we can define \textit{locally metallic Riemannian manifold}, as follows:

\begin{definition}
If $(\overline{M},\overline{g}, J)$ is a metallic Riemannian manifold and $J$ is parallel with respect to the Levi-Civita connection $\overline{\nabla}$ on $\overline{M}$ (i.e. $\overline{\nabla}J=0$), then we call $(\overline{M},\overline{g}, J)$ {\it a locally metallic Riemannian manifold}.
\end{definition}

As in the case of submanifolds in Riemannian manifolds with Golden structure (\cite{Hr2}), it is easy to verify the following relations regarding the covariant derivatives of  components of the $\Sigma =(P, g, u_{\alpha}, \xi _{\alpha},(a_{\alpha\beta})_{r})$-structure, induced on $M$ by the metallic Riemannian structure $(\overline{g}, J)$:

\begin{theorem} If $M$ is an $n$-dimensional submanifold of codimension $r$ in a locally metallic Riemannian manifold $(\overline{M},\overline{g}, J)$, then the $\Sigma =(P, g, u_{\alpha}, \xi _{\alpha},(a_{\alpha\beta})_{r})$-structure, induced on $M$ by the metallic Riemannian structure $(\overline{g}, J)$, has the following properties:

\begin{equation}\label{e17}
(\nabla_{X}P)(Y)=\sum_{\alpha=1}^{r}h_{\alpha}(X, Y)\xi _{\alpha}+\sum_{\alpha=1}^{r}u_{\alpha}(Y)A_{\alpha}X,
\end{equation}

\begin{equation}\label{e18}
(\nabla_{X}u_{\alpha})(Y)=-h_{\alpha}(X, PY)+\sum_{\beta=1}^{r}[l_{\alpha\beta}(X)u_{\beta}(Y)+a_{\beta\alpha}h_{\beta}(X, Y)],
\end{equation}

\begin{equation}\label{e19}
 \nabla_{X}\xi _{\alpha}=-P(A_{\alpha}X)+\sum_{\beta=1}^{r}a_{\alpha\beta}A_{\beta}X+\sum_{\beta=1}^{r}l_{\alpha\beta}(X)\xi _{\beta},
 \end{equation}

 \begin{equation}\label{e20}
 X(a_{\alpha\beta})=-u_{\alpha}(A_{\beta}X)-u_{\beta}(A_{\alpha}X)+\sum_{\gamma=1}^{r}[a_{\gamma \beta}l_{\alpha\gamma }(X)+ a_{\alpha\gamma }l_{\beta \gamma}(X)],
\end{equation}
for any $X, Y \in \mathcal{X}(M)$.
\end{theorem}

\begin{proof}
 Using the Gauss and Weingarten formulae in $(\overline{\nabla}_{X} J)Y = 0 $, for any $X, Y \in \mathcal{X}(M)$ and identifying the tangential components (and the normal components, respectively), we obtain the equalities (\ref{e17}) and (\ref{e18}). Moreover, from
\[ \overline{\nabla}_{X}(J N_{\alpha})=\nabla_{X}\xi_{\alpha} -\sum_{\beta=1}^{r} a_{\alpha\beta}A_{\beta}X+\sum_{\beta=1}^{r}
[X(a_{\alpha\beta})+ h_{\beta}(X,\xi_{\alpha})+\sum_{\gamma=1}^{r} a_{\alpha\gamma} l_{\gamma \beta}(X)]N_{\beta} \] and
\[J(\overline{\nabla}_{X}N_{\alpha})=-P(A_{\alpha}X)+ \sum_{\beta=1}^{r} l_{\alpha\beta}(X)\xi_{\beta}-\sum_{\beta=1}^{r} [u_{\beta}(A_{\alpha}X)-\sum_{\gamma=1}^{r} a_{\gamma\beta}l_{\alpha\gamma}(X)]N_{\beta},\]
for any $X \in \mathcal{X}(M)$ and using $(\overline{\nabla}_{X}J) N_{\alpha}=0$ we obtain (\ref{e19}) and (\ref{e20}) by identifying the tangential and the normal components, respectively.
\end{proof}

\begin{remark}
If $M$ is an invariant submanifold in the locally metallic Riemannian manifold $(\overline{M},\overline{g}, J)$,
then $P$ is parallel with respect to the Levi-Civita connection $\nabla$ on $M$, where $P$ is the $(1, 1)$-tensor field
on the Riemannian manifold $(M, g)$ defined in (\ref{e100})(i).
\end{remark}

\section{Examples of structures induced on submanifolds by metallic Riemannian structures}
\label{sec:4}
\normalfont

We assume that the ambient space is $E^{2a+b}$ ($a,b \geq 1$ integer numbers) and for any point of $E^{2a+b}$ we have its
coordinates:
$$(x^{1},...,x^{a},y^{1},...,y^{a},z^{1},...,z^{b}):=(x^{i},y^{i},z^{j}),$$
where $i\in \{1,...,a\}$ and $j \in \{1,...,b\}$. The tangent space $T_{x}(E^{2a+b})$ is isomorphic with $E^{2a+b}$.
For $\lambda \in \{-1, 1\}$ we can construct two metallic structures on $E^{2a+b}$, $J_{\lambda}:E^{2a+b}\rightarrow E^{2a+b}$, given by:
\begin{equation} \label{e33}
J_{\lambda}\left(\frac{\partial}{\partial x^{i}},\frac{\partial}{\partial y^{i}},\frac{\partial}{\partial z^{j}}\right)=
\left(\frac{p}{2}\frac{\partial}{\partial x^{i}} + \frac{\lambda \sqrt{\Delta}}{2}\frac{\partial}{\partial y^{i}},\frac{p}{2}\frac{\partial}{\partial y^{i}} +
 \frac{\lambda \sqrt{\Delta}}{2}\frac{\partial}{\partial x^{i}},\left(\frac{p}{2} + \frac{\lambda \varepsilon_{j} \sqrt{\Delta}}{2}\right)\frac{\partial}{\partial z^{j}}\right),
\end{equation}
where $\varepsilon_{j} \in \{-1, 1\}$, $i\in \{1,...,a \}$, $j\in \{1,...,b \}$ and $\Delta = p^{2}+4q$.

We easily find that $J_{\lambda}^{2}=p J_{\lambda}+q  I$ and $(J_{\lambda}, <\cdot,\cdot>)$ is a metallic Riemannian structure on $E^{2a+b}$.

Thus, we obtain two metallic structures:  $J_{+1}$ (for $\lambda = 1$) and $J_{-1}$ (for $\lambda = -1$).
Moreover, for $\lambda\varepsilon_{j} = 1$, we obtain the following form of $J_{\lambda}$:
\begin{equation} \label{e34}
J_{\lambda}\left(\frac{\partial}{\partial x^{i}},\frac{\partial}{\partial y^{i}},\frac{\partial}{\partial z^{j}}\right)=
\left(\frac{p}{2}\frac{\partial}{\partial x^{i}} + \lambda \frac{\sqrt{\Delta}}{2}\frac{\partial}{\partial y^{i}},
\frac{p}{2}\frac{\partial}{\partial y^{i}} + \lambda \frac{\sqrt{\Delta}}{2}\frac{\partial}{\partial x^{i}},
\sigma_{p,q}\frac{\partial}{\partial z^{j}}\right),
 \end{equation}
 and for $\lambda\varepsilon_{j} = -1$, we obtain:
 \begin{equation} \label{e340}
J_{\lambda}\left(\frac{\partial}{\partial x^{i}},\frac{\partial}{\partial y^{i}},\frac{\partial}{\partial z^{j}}\right)=
\left(\frac{p}{2}\frac{\partial}{\partial x^{i}} + \lambda \frac{\sqrt{\Delta}}{2}\frac{\partial}{\partial y^{i}},
\frac{p}{2}\frac{\partial}{\partial y^{i}} + \lambda \frac{\sqrt{\Delta}}{2}\frac{\partial}{\partial x^{i}},
\overline{\sigma}_{p,q}\frac{\partial}{\partial z^{j}}\right),
 \end{equation}
where $\overline{\sigma}_{p,q}=p-\sigma$. Now, we can construct a $\Sigma $-metallic Riemannian structure on the sphere
$S^{2a+b-1}(R)\hookrightarrow E^{2a+b}$. The equation of sphere $S^{2a+b-1}(R)$ is:
\begin{equation} \label{e35}
\sum_{i=1}^{a}(x^{i})^{2}+\sum_{i=1}^{a}(y^{i})^{2}+\sum_{j=1}^{b}(z^{j})^{2}=R^{2},
\end{equation}
where $R$ is its radius and $(x^{1},...,x^{a},y^{1},...,y^{a},z^{1},...,z^{b}):=(x^{i},y^{i},z^{j})$ are the
coordinates of any point of $S^{2a+b-1}(R)$. We denote by:
\begin{equation} \label{e36}
\sum_{i=1}^{a}(x^{i})^{2}=r_{1}^{2},\:
\sum_{i=1}^{a}(y^{i})^{2}=r_{2}^{2},\:
\sum_{j=1}^{b}(z^{j})^{2}=r_{3}^{2}
\end{equation}
and we obtain $r_{1}^{2}+r_{2}^{2}+r_{3}^{2}=R^{2}$.

An unit normal vector field $N$ on the sphere $S^{2a+b-1}(R)$ is given by:
\begin{equation} \label{e37}
N= \frac{1}{R}(x^{1},...,x^{a},y^{1},...,y^{a},z^{1},...,z^{b}).
\end{equation}

Thus:
\begin{equation} \label{e38}
J_{\lambda} N= \frac{1}{R}\left(\left(\frac{p}{2}x^{i} +\lambda \frac{\sqrt{\Delta}}{2}y^{i}\right),\left(\frac{p}{2}y^{i} +\lambda \frac{\sqrt{\Delta}}{2}x^{i}\right),\sigma_{p,q}z^{j}\right),
\end{equation}
where $\Delta = p^{2}+4q$ and $\lambda \in\{-1, 1\}$.

For a vector field $X$ on $S^{2a+b-1}(R)$ we use the
following notation:
\[X=(X^{1},...,X^{a},Y^{1},...,Y^{a},Z^{1},...,Z^{b}):=(X^{i},Y^{i},Z^{j})\]
and from $<X,N>=0$ we obtain:
\begin{equation} \label{e39}
\sum_{i=1}^{a}\left(x^{i}X^{i}+y^{i}Y^{i}\right)+\sum_{j=1}^{b}z^{j}Z^{j}=0.
\end{equation}

If we decompose $J_{\lambda} N$ into tangential and normal components at the sphere $S^{2a+b-1}(R)$, for $\mathcal{A} = < J_{\lambda} N,N >$, we obtain:
\begin{equation} \label{e40}
 J_{\lambda} N= \xi+  \mathcal{A} N.
\end{equation}

In the followings we consider $\lambda\varepsilon_{j} = 1$. Thus, we obtain:
\begin{equation} \label{e41}
\mathcal{A}=\frac{1}{R^{2}}\left[\frac{p}{2}(r_{1}^{2}+r_{2}^{2}) +\lambda \sqrt{\Delta}\sum_{i=1}^{a}x^{i}y^{i} + \sigma_{p,q} r_{3}^{2}\right].
\end{equation}

The tangential component of $J_{\lambda} N$ at the sphere $S^{2a+b-1}(R)$ has the form:
\begin{equation} \label{e42}
\xi=\frac{1}{R}\left( \left(\frac{p}{2}-\mathcal{A}\right) x^{i} + \lambda\frac{\sqrt{\Delta}}{2}y^{i},
\left(\frac{p}{2}-\mathcal{A}\right) y^{i} + \lambda \frac{\sqrt{\Delta}}{2}x^{i},
\left(\frac{p}{2}-\mathcal{A} + \frac{\sqrt{\Delta}}{2}\right)z^{j} \right).
\end{equation}

From $u(X)=<X,\xi>$, we obtain:
\begin{equation} \label{e43}
u(X)=\lambda \frac{\sqrt{\Delta}}{2R}\left[\sum_{i=1}^{a}(y^{i}X^{i}+x^{i}Y^{i}) +\lambda \sum_{j=1}^{b}z^{j}Z^{j}\right].
\end{equation}

The decomposition of $J X$ into tangential and normal components at the sphere
$S^{2a+b-1}(R)$ (where $X$ is a vector field on $S^{2a+b-1}(R)$), is given by:
\begin{equation} \label{e44}
J_{\lambda} X = P X + u(X)N.
\end{equation}

Thus, we obtain:
\begin{equation} \label{e45}
P X=\left(\frac{p}{2}X^{i} + \lambda \frac{\sqrt{\Delta}}{2}Y^{i}-\frac{u(X)}{R}x^{i},
\frac{p}{2}Y^{i} + \lambda \frac{\sqrt{\Delta}}{2}X^{i}-\frac{u(X)}{R}y^{i},
\left(\sigma_{p,q}Z^{j}-\frac{u(X)}{R}z^{j}\right)\right),
\end{equation}
for any tangent vector $X:=(X^{i},Y^{i},Z^{j})$
at the sphere $S^{2a+b-1}(R)$ in a point $(x^{i},y^{i},z^{j})$,
for $i \in \{1,...,a\}$ and $j \in \{1,...,b\}$.

Therefore, from the relations (\ref{e41}), (\ref{e42}), (\ref{e43}) and (\ref{e45}) we
obtain the $\Sigma$-structure induced on the sphere
$S^{2a+b-1}(R)$ of codimension $1$ in the Euclidean space $E^{2a+b}$
by $J_{\lambda}$ from $E^{2a+b}$, given by $(P,<\cdot,\cdot>,\xi,u,\mathcal{A})$.

In particular, for $\lambda\varepsilon_{j} = 1$ (where $\lambda \in \{-1, 1\}$ and $\varepsilon_{j} \in \{-1, 1\}$), we obtain:
\begin{enumerate}
\item for $p=q=1$, two Golden structures, given by:
\begin{equation} \label{e46}
\Phi_{\lambda}\left(\frac{\partial}{\partial x^{i}},\frac{\partial}{\partial y^{i}},\frac{\partial}{\partial z^{j}}\right)=
\left(\frac{1}{2}\frac{\partial}{\partial x^{i}} +\lambda \frac{\sqrt{5}}{2}\frac{\partial}{\partial y^{i}},\frac{1}{2}\frac{\partial}{\partial y^{i}}+\lambda \frac{\sqrt{5}}{2}\frac{\partial}{\partial x^{i}}, \phi \frac{\partial}{\partial z^{j}}\right)
\end{equation}

\item for $p=2, q=1$, two Silver structures, given by:
\begin{equation} \label{e47}
Ag_{\lambda}\left(\frac{\partial}{\partial x^{i}},\frac{\partial}{\partial y^{i}},\frac{\partial}{\partial z^{j}}\right)=
\left(\frac{\partial}{\partial x^{i}} +\lambda \sqrt{2}\frac{\partial}{\partial y^{i}},
\frac{\partial}{\partial y^{i}} +\lambda \sqrt{2}\frac{\partial}{\partial x^{i}},
\sigma_{Ag} \frac{\partial}{\partial z^{j}}\right)
\end{equation}

\item for $p=3, q=1$, two Bronze structures, given by:
\begin{equation} \label{e48}
Br_{\lambda}\left(\frac{\partial}{\partial x^{i}},\frac{\partial}{\partial y^{i}},\frac{\partial}{\partial z^{j}}\right)=
\left(\frac{3}{2}\frac{\partial}{\partial x^{i}} +\lambda \frac{\sqrt{13}}{2}\frac{\partial}{\partial y^{i}},\frac{3}{2}\frac{\partial}{\partial y^{i}} +
\lambda \frac{\sqrt{13}}{2}\frac{\partial}{\partial x^{i}},\sigma_{Br}\frac{\partial}{\partial z^{j}}\right)
\end{equation}

\item for $p=1, q=2$, two Copper structures, given by:
\begin{equation} \label{e49}
Cu_{\lambda}\left(\frac{\partial}{\partial x^{i}},\frac{\partial}{\partial y^{i}},\frac{\partial}{\partial z^{j}}\right)=
\left(\frac{1}{2}\frac{\partial}{\partial x^{i}} +\lambda \frac{3}{2}\frac{\partial}{\partial y^{i}},
\frac{1}{2}\frac{\partial}{\partial y^{i}} +\lambda \frac{3}{2}\frac{\partial}{\partial x^{i}},
\sigma_{Cu}\frac{\partial}{\partial z^{j}}\right)
\end{equation}

\item for $p=1, q=3$, two Nickel structures, given by:
\begin{equation} \label{e50}
Ni_{\lambda}\left(\frac{\partial}{\partial x^{i}},\frac{\partial}{\partial y^{i}},\frac{\partial}{\partial z^{j}}\right)=
\left(\frac{1}{2}\frac{\partial}{\partial x^{i}} +\lambda \frac{\sqrt{13}}{2}\frac{\partial}{\partial y^{i}},
\frac{1}{2}\frac{\partial}{\partial y^{i}} + \lambda \frac{\sqrt{13}}{2}\frac{\partial}{\partial x^{i}},\sigma_{Ni}\frac{\partial}{\partial z^{j}}\right).
\end{equation}
\end{enumerate}

The induced structures on the sphere $S^{2a+b-1}(R)$ of codimension $1$ in the Euclidean space $E^{2a+b}$ are as follows:

\begin{enumerate}
\item for $p=q=1$, we obtain the $\sum$-Golden structure, given by:

$$ \mathcal{A}=\frac{1}{R^{2}}\left[\frac{1}{2}(r_{1}^{2}+r_{2}^{2}) +\lambda \sqrt{5}\sum_{i=1}^{a}x^{i}y^{i} + \phi r_{3}^{2}\right] $$
$$ \xi=\frac{1}{R}\left(\left(\frac{1}{2}-\mathcal{A}\right) x^{i} + \lambda \frac{\sqrt{5}}{2}y^{i},\left(\frac{1}{2}-\mathcal{A}\right) y^{i} + \lambda \frac{\sqrt{5}}{2}x^{i},\left(\frac{1}{2}-\mathcal{A} + \frac{\sqrt{5}}{2}\right)z^{j}\right) $$
$$ u(X)=\lambda \frac{\sqrt{5}}{2R}\left[\sum_{i=1}^{a}(y^{i}X^{i}+x^{i}Y^{i}) + \lambda \sum_{j=1}^{b}z^{j}Z^{j}\right] $$
$$ P X=\left(\frac{1}{2}X^{i} + \lambda \frac{\sqrt{5}}{2}Y^{i}-\frac{u(X)}{R}x^{i},\: \frac{1}{2}Y^{i} + \lambda \frac{\sqrt{5}}{2}X^{i}-\frac{u(X)}{R}y^{i}, \: \phi Z^{j}-\frac{u(X)}{R}z^{j}\right),$$
induced by the structure defined in (\ref{e46});

\item for $p=2, q=1$, we obtain the $\sum$-Silver structures, given by:
$$ \mathcal{A}=\frac{1}{R^{2}}\left[(r_{1}^{2}+r_{2}^{2}) + 2\lambda \sqrt{2}\sum_{i=1}^{a}x^{i}y^{i} + \sigma_{Ag} r_{3}^{2}\right]$$
$$ \xi=\frac{1}{R}\left(\left(1-\mathcal{A}\right) x^{i} + \lambda \sqrt{2}y^{i},\left(1-\mathcal{A}\right) y^{i} + \lambda \sqrt{2}x^{i},(1-\mathcal{A} + \sqrt{2})z^{j}\right) $$
$$ u(X)=\lambda \frac{\sqrt{2}}{R}\left[\sum_{i=1}^{a}(y^{i}X^{i}+x^{i}Y^{i}) + \lambda \sum_{j=1}^{b}z^{j}Z^{j}\right] $$
$$ P X=\left(X^{i} + \lambda \sqrt{2}Y^{i}-\frac{u(X)}{R}x^{i}, \: Y^{i} + \lambda \sqrt{2}X^{i}-\frac{u(X)}{R}y^{i}, \:\sigma_{Ag} Z^{j}-\frac{u(X)}{R}z^{j}\right),$$
induced by the structure defined in (\ref{e47});

\item for $p=3, q=1$, we obtain the $\sum$-Bronze structures, given by:
$$ \mathcal{A}=\frac{1}{R^{2}}\left[\frac{3}{2}(r_{1}^{2}+r_{2}^{2}) + \lambda \sqrt{13}\sum_{i=1}^{a}x^{i}y^{i} + \sigma_{Br} r_{3}^{2}\right] $$
$$ \xi=\frac{1}{R}\left(\left(\frac{3}{2}-\mathcal{A}\right) x^{i} + \lambda \frac{\sqrt{13}}{2}y^{i},\left(\frac{3}{2}-\mathcal{A}\right) y^{i} + \lambda \frac{\sqrt{13}}{2}x^{i},\left(\frac{3}{2}-\mathcal{A} + \frac{\sqrt{13}}{2}\right)z^{j}\right) $$
$$ u(X)=\lambda \frac{\sqrt{13}}{2R}\left[\sum_{i=1}^{a}(y^{i}X^{i}+x^{i}Y^{i}) + \lambda \sum_{j=1}^{b}z^{j}Z^{j}\right]$$
$$ P X=\left(\frac{3}{2}X^{i} + \lambda \frac{\sqrt{13}}{2}Y^{i}-\frac{u(X)}{R}x^{i}, \: \frac{3}{2}Y^{i} + \lambda \frac{\sqrt{13}}{2}X^{i}-\frac{u(X)}{R}y^{i}, \: \sigma_{Br} Z^{j}-\frac{u(X)}{R}z^{j}\right)$$
induced by the structure defined in (\ref{e48});

\item for $p=1, q=2$, we obtain the $\sum$-Copper structures, given by:
$$ \mathcal{A}=\frac{1}{R^{2}}\left[\frac{1}{2}(r_{1}^{2}+r_{2}^{2}) + 3 \lambda\sum_{i=1}^{a}x^{i}y^{i} + \sigma_{Cu} r_{3}^{2}\right] $$
$$ \xi=\frac{1}{R}\left(\left(\frac{1}{2}-\mathcal{A}\right) x^{i} +\lambda \frac{3}{2}y^{i},\left(\frac{1}{2}-\mathcal{A}\right) y^{i} +\lambda \frac{3}{2}x^{i},\left(\frac{1}{2}-\mathcal{A} + \frac{3}{2}\right)z^{j}\right) $$
$$ u(X)=\lambda \frac{3}{2R}\left[\sum_{i=1}^{a}(y^{i}X^{i}+x^{i}Y^{i}) +\lambda \sum_{j=1}^{b}z^{j}Z^{j}\right] $$
$$ P X=\left(\frac{1}{2}X^{i} +\lambda \frac{3}{2}Y^{i}-\frac{u(X)}{R}x^{i}, \: \frac{1}{2}Y^{i} +\lambda \frac{3}{2}X^{i}-\frac{u(X)}{R}y^{i},\: \sigma_{Cu} Z^{j}-\frac{u(X)}{R}z^{j}\right),$$
induced by the structure defined in (\ref{e49});

\item for $p=1, q=3$, we obtain the $\sum$-Nickel structures, given by:
$$ \mathcal{A}=\frac{1}{R^{2}}\left[\frac{1}{2}(r_{1}^{2}+r_{2}^{2}) + \lambda \sqrt{13}\sum_{i=1}^{a}x^{i}y^{i} + \sigma_{Ni} r_{3}^{2}\right] $$
$$ \xi=\frac{1}{R}\left(\left(\frac{1}{2}-\mathcal{A}\right) x^{i} +\lambda \frac{\sqrt{13}}{2}y^{i},\left(\frac{1}{2}-\mathcal{A}\right) y^{i} +\lambda \frac{\sqrt{13}}{2}x^{i},
\left(\frac{1}{2}-\mathcal{A} + \frac{\sqrt{13}}{2}\right)z^{j}\right)$$
$$ u(X)=\lambda \frac{\sqrt{13}}{2R}\left[\sum_{i=1}^{a}(y^{i}X^{i}+x^{i}Y^{i}) +\lambda \sum_{j=1}^{b}z^{j}Z^{j}\right] $$
$$ P X=\left(\frac{1}{2}X^{i} +\lambda \frac{\sqrt{13}}{2}Y^{i}-\frac{u(X)}{R}x^{i}, \: \frac{1}{2}Y^{i} + \lambda \frac{\sqrt{13}}{2}X^{i}-\frac{u(X)}{R}y^{i}, \: \sigma_{Ni} Z^{j}-\frac{u(X)}{R}z^{j}\right),$$
induced by the structure defined in (\ref{e50}).

\end{enumerate}

In conclusion, we can obtain $\Sigma$-metallic structures $(P,<\cdot,\cdot>,\xi,u,\mathcal{A})$, induced on the sphere
$S^{2a+b-1}(R)$ of codimension $1$ in the Euclidean space $E^{2a+b}$ by the metallic structure $J_{\lambda}$ from $E^{2a+b}$, for various positive integer values of $p$ and $q$.

\section{On the normality of $\Sigma $-metallic Riemannian structures}
\label{sec:3}
\normalfont

Let $N_{J}$ be the Nijenhuis torsion tensor field of $J$, defined by:
\begin{equation}\label{e15}
N_{J}(X, Y):=[JX, JY]+J^{2}[X, Y]-J[JX, Y]-J[X, JY],
\end{equation}
for any $X, Y \in \mathcal{X}(\overline{M})$. It is known that the Nijenhuis tensor of $J$ verifies (\cite{Yano}):
\begin{equation}\label{e16}
N_{J}(X,Y)=(\overline{\nabla}_{JX}J)(Y)-(\overline{\nabla}_{JY}J)(X)-J[(\overline{\nabla}_{X}J)(Y)-(\overline{\nabla}_{Y}J)(X)],
\end{equation}
for any $X, Y \in \mathcal{X}(\overline{M})$.

Thus, we remark that if the metallic structure $J$ on $\overline{M}$ is parallel with respect to $\overline{\nabla}$, then $N_{J}=0$.

     We shall define a normal $\Sigma $-metallic Riemannian structure on the submanifold $M$ in the metallic Riemannian manifold $(\overline{M},\overline{g}, J)$.

\begin{definition}
A $\Sigma $-metallic Riemannian structure defined on a submanifold $(M,g)$ of codimension $r$ in a Riemannian manifold $(\overline{M}, \overline{g})$  is said to be normal if:
\begin{equation}\label{e21}
N_{P}=2\sum_{\alpha=1}^{r}du_{\alpha}\otimes \xi_{\alpha}.
\end{equation}
\end{definition}

In the following considerations, let $M$ be a submanifold in the locally metallic Riemannian manifold $(\overline{M},\overline{g}, J)$ and let $\nabla$ be the Levi-Civita connection determined by the induced metric $g$ on $M$.

We denote by
$B_{\alpha}=:PA_{\alpha}-A_{\alpha}P$ and remark that $g(B_{\alpha}X, Y)=-g(B_{\alpha}Y, X)$, for any $X, Y \in \mathcal{X}(M)$.

\begin{theorem} If $M$ is a submanifold of codimension $r$ in a locally metallic Riemannian manifold $(\overline{M},\overline{g}, J)$
and $\Sigma =(P, g, u_{\alpha }, \xi _{\alpha }, (a_{\alpha \beta})_{r})$ is the structure induced on the Riemannian manifold $(M, g)$,
then the Nijenhuis torsion tensor field of the $(1, 1)$-tensor field $P$ and the $1$-forms $u_{\alpha}$  verify the equalities:
\begin{equation}\label{e22}
N_{P}(X, Y)=\sum_{\alpha=1}^{r}[g(X, \xi _{\alpha})B_{\alpha}Y-g(Y, \xi _{\alpha})B_{\alpha}X-g(B_{\alpha}X, Y)\xi _{\alpha}],
\end{equation}
\begin{equation}\label{e23}
2du_{\alpha}(X, Y)=-g(B_{\alpha}X, Y)+\sum_{\beta=1}^{r}[l_{\alpha\beta}(X)g(Y, \xi _{\beta})-l_{\alpha\beta}(Y)g(X, \xi _{\beta})],
\end{equation}
for any $X,Y \in \mathcal{X}(M)$, where $l_{\alpha\beta}$ are the coefficients of the normal connection $\nabla^{\bot}$ in the normal bundle $T^{\bot}(M)$. \\
\end{theorem}

\begin{proof}
Using:
\[N_{P}(X, Y)=\sum_{\alpha=1}^{r}[g(A_{\alpha}PX, Y)-g(A_{\alpha}PY, X)]\xi _{\alpha}+
\]
\[+ \sum_{\alpha=1}^{r}[g(Y, \xi_{\alpha})A_{\alpha}(PX)-g(X, \xi_{\alpha})A_{\alpha}(PY)-\sum_{\alpha=1}^{r}[g(Y, \xi _{\alpha})P(A_{\alpha}X)-g(X, \xi _{\alpha})P(A_{\alpha}Y)]
\]
and from $g(A_{\alpha}PY, X)=g(PY, A_{\alpha}X)=g(Y, P(A_{\alpha}X))=g(P(A_{\alpha}X), Y)$ we obtain that the Nijenhuis torsion tensor field of the $(1, 1)$-tensor field $P$ from the $\Sigma $-metallic Riemannian structure has the form:
\[
N_{P}(X, Y)=-\sum_{\alpha=1}^{r}g((PA_{\alpha}-A_{\alpha}P)(X), Y)\xi _{\alpha}-
\]
\[
-\sum_{\alpha=1}^{r}g(Y, \xi _{\alpha})(PA_{\alpha}-A_{\alpha}P)(X)+\sum_{\alpha=1}^{r}g(X, \xi_{\alpha})(PA_{\alpha}-A_{\alpha}P)(Y),
\]
for any $X, Y \in \mathcal{X}(M)$ which implies (\ref{e22}).

From $2du_{\alpha}(X, Y)=X(u_{\alpha}(Y))-Y(u_{\alpha}(X))-u_{\alpha}([X, Y])$ we have:
\[
2du_{\alpha}(X, Y)=(\nabla_{X}u_{\alpha})(Y)-(\nabla_{Y}u_{\alpha})(X),
\]
for any $X, Y \in \mathcal{X}(M)$ and using (\ref{e20}) we obtain:
\[
2du_{\alpha}(X, Y)=g(A_{\alpha}Y, PX)-g(A_{\alpha}X, PY)+\sum_{\beta=1}^{r}[g(A_{\beta}X, Y)-g(A_{\beta}Y, X)]a_{\alpha\beta}+
\]
\[+\sum_{\beta=1}^{r}[g(Y, \xi _{\beta})l_{\alpha\beta}(X)-g(X, \xi _{\beta})l_{\alpha\beta}(Y)].\]

From $g(A_{\beta}X, Y)=g(A_{\beta}Y, X)$ we have:
\[
g(A_{\alpha}X, PY)-g(A_{\alpha}Y, PX)=g((PA_{\alpha}-A_{\alpha}P)(X), Y)=g(B_{\alpha}X, Y)
\]
and we obtain that the $1$-forms $u_{\alpha}$ of the $\Sigma $-metallic Riemannian structure induced on $M$ verify the equality (\ref{e23}).
\end{proof}

\begin{remark}
Let $M$ be a submanifold of codimension $r$ in a locally metallic Riemannian manifold $(\overline{M},\overline{g}, J)$ and let
$\Sigma =(P, g, u_{\alpha }, \xi _{\alpha }, (a_{\alpha \beta})_{r})$ be the structure induced on the Riemannian manifold $(M, g)$.
If the $(1, 1)$-tensor field $P$ of the $\Sigma $-metallic Riemannian structure on $M$ commutes with the Weingarten operators
$A_{\alpha}$, for any $\alpha \in \{1,...,r\}$ (i.e. $B_{\alpha}=0$), then the Nijenhuis torsion tensor field of $P$ vanishes on $M$ (i.e. $N_{P}=0)$.
\end{remark}

\begin{theorem}
If $M$ is a submanifold of codimension $r$ in the locally metallic Riemannian manifold $(\overline{M},\overline{g}, J)$, then:
\begin{equation}\label{e24}
N_{P}(X, Y)-2\sum_{\alpha=1}^{r}du_{\alpha}(X, Y)\xi _{\alpha}= \sum_{\alpha=1}^{r}[g(X, \xi _{\alpha})B_{\alpha}Y-g(Y, \xi _{\alpha})B_{\alpha}X]-
\end{equation}
\[-\sum_{\alpha=1}^{r}\sum_{\beta=1}^{r}[l_{\alpha\beta}(X)g(Y, \xi _{\beta})-l_{\alpha\beta}(Y)g(X, \xi _{\beta})]\xi _{\alpha}, \]
for any $X, Y\in \mathcal{X}(M)$, where $l_{\alpha\beta}$ are the coefficients of the normal connection in the normal bundle $T^{\bot}(M)$.
\end{theorem}

\begin{remark}
If the $\Sigma $-metallic Riemannian structure induced on $M$ is normal and the normal connection $\nabla^{\bot}$ of $M$ vanishes identically (i.e. $l_{\alpha \beta}=0$), then we obtain the equality:
\begin{equation}\label{e25}
\sum_{\alpha=1}^{r}g(X, \xi _{\alpha})(PA_{\alpha}-A_{\alpha}P)(Y)= \sum_{\alpha=1}^{r}g(Y, \xi _{\alpha})(PA_{\alpha}-A_{\alpha}P)(X),
\end{equation}
and this equality does not depend on the choice of a basis in the normal space $T_{x}^{\bot}(M)$, for any $x \in M$.
\end{remark}

\begin{theorem}
Let $M$ be a submanifold of codimension $r$ in a locally metallic Riemannian manifold $(\overline{M},\overline{g},J)$.
If the normal connection $\nabla^{\bot}$ vanishes identically on the normal bundle $T^{\bot}(M)$ and
the $(1, 1)$-tensor field $P$ of the $\Sigma $-metallic Riemannian structure induced on $M$ commutes with
every Weingarten operator $A_{\alpha}$, then the induced $\Sigma $-metallic Riemannian structure on $M$ is normal.
\end{theorem}

\begin{remark} The matrix $\mathcal{A}:=(a_{\alpha\beta})_{r}$ of the structure $\Sigma$ induced on an in\-va\-riant submanifold $M$ by the metallic structure $J$ from the Riemannian manifold $(\overline{M},\overline{g})$ is a metallic matrix, that is a matrix which verifies:
$$ \mathcal{A}^{2}=p \mathcal{A}+ q I_{r}, $$ where $I_{r}$  is the identically matrix of order $r$ (\cite{Hr4}).

Conversely, if $\mathcal{A}:=(a_{\alpha\beta})_{r}$ is a metallic matrix, then
\[
\sum_{\gamma=1}^{r}a_{\alpha\gamma }a_{\gamma \beta}=p a_{\alpha\beta} + q \delta _{\alpha\beta}
\]
and from (\ref{e10})(i) we obtain $u_{\beta}(\xi_{\alpha})=0$, for any $\alpha, \beta \in \{1,..., r\}$,
which implies that $P^{2}X=pPX+qX$ and $JX=PX$, for any $X \in \mathcal{X}(M)$.
Thus, the Riemannian submanifold $(M,g)$ is an invariant submanifold in the metallic Riemannian manifold $(\overline{M},\overline{g}, J)$.
\end{remark}

\begin{theorem}
Let $M$ be a submanifold of codimension $r\geq 2$ in a locally metallic Riemannian manifold $(\overline{M},\overline{g}, J)$. If the normal connection $\nabla^{\bot}$
vanishes identically on the normal bundle $T^{\bot}(M)$ and $M$ is a non-invariant submanifold with respect to the metallic structure $J$, then the vector fields $\{\xi _{1},..., \xi _{r}\}$ are linearly independent.
\end{theorem}
\begin{proof}
From (\ref{e10})(i) we have:
\[
g(\xi _{\alpha}, \xi _{\beta})=q \delta _{\alpha\beta}+p a_{\alpha\beta}-\sum_{\gamma=1}^{r}a_{\alpha\gamma }a_{\gamma \beta}.
\]
Let $x_{1}, ...,x_{r}$ be real numbers with the property $x_{1}\xi_{1}+...+x_{r}\xi_{r}=0$ in any point $x \in M$.
Applying $g(\xi _{\alpha}, \cdot)$ in the last equality ($\alpha \in \{1,...,r\}$), we obtain a linear system, given by:
\begin{equation} \label{e26}
\sum_{i=1}^{r}x_{i}\Gamma_{ij}=0
\end{equation}
for any $j \in \{1,...,r\}$, where $\Gamma_{ii}= q+pa_{ii}-\sum_{\gamma }a_{i \gamma }^{2}$ and $\Gamma_{ij}= pa_{ij}-\sum_{\gamma }a_{i\gamma }a_{\gamma j}$ for $i \neq j$ and $i, j \in \{1,...,r\}$.

The determinant of the matrix of the linear system (\ref{e26}) is the determinant of the matrix $\mathcal{B} = q I_{r}+p \mathcal{A}-\mathcal{A}^{2}$, where $\mathcal{A} = (a_{\alpha\beta})_r$. Using the last remark, we observe that the determinant of the matrix $\mathcal{B}$ is not zero when  $M$ is a non-invariant submanifold with respect to the metallic structure $J$.
Therefore, the linear system of equations (\ref{e26}) has the unique solution $x_{1}=...=x_{r}=0$.
Thus, the vector fields $\{\xi_{1},...,\xi_{r}\}$ are linearly independent.
\end{proof}

\begin{theorem}
Let $M$ be a submanifold  of codimension $r \geq 1$ in a locally metallic Riemannian manifold $(\overline{M},\overline{g}, J)$.
If the normal connection $\nabla^{\bot}$ vanishes identically on the normal bundle $T^{\bot}(M)$ and $M$ is a non-invariant
submanifold with respect to the metallic structure $J$, then the induced $\Sigma $-metallic Riemannian structure on $M$
is normal if and only if the $(1, 1)$-tensor field $P$ commutes with every Weingarten operator $A_{\alpha}$, for any $\alpha \in \{1,..., r\}$.
\end{theorem}
\begin{proof}
We assume that the induced $\Sigma$-metallic Riemannian structure on $M$ is normal and we prove that the $(1, 1)$-tensor field $P$
commutes with every Weingarten operator $A_{\alpha}$ (i.e. $PA_{\alpha}=A_{\alpha}P$), for any $\alpha \in \{1,...,r\}$.

If the normal connection vanishes identically (thus $l_{\alpha\beta}=0$), then:
\begin{equation} \label{e27}
\sum_{\alpha=1}^{r}g(X, \xi _{\alpha})B_{\alpha}Y=\sum_{\alpha=1}^{r}g(Y, \xi _{\alpha})B_{\alpha}X.
\end{equation}

Multiplying the equality (\ref{e27}) by $Z \in \mathcal{X}(M)$, we obtain:
\begin{equation} \label{e28}
\sum_{\alpha=1}^{r}g(X, \xi _{\alpha})g(B_{\alpha} Y, Z)=\sum_{\alpha=1}^{r}g(Y, \xi _{\alpha})g(B_{\alpha} X, Z),
\end{equation}
for any $X, Y, Z \in \mathcal{X}(M)$.

Inverting $Y$ by $Z$ in the last equality and using that $B_{\alpha}$ is $g$-skew symmetric (i.e. $g(B_{\alpha}Y, Z)+g(B_{\alpha}Z, Y)=0$),
by summing these relations we obtain:
\[
\sum_{\alpha=1}^{r}g(Y, \xi _{\alpha})g(B_{\alpha}X, Z)+\sum_{\alpha=1}^{r}g(Z, \xi_{\alpha})g(B_{\alpha}X, Y)=0.
\]
Therefore:
\begin{equation} \label{e29}
\sum_{\alpha=1}^{r}g(Y, \xi _{\alpha})B_{\alpha}X+\sum_{\alpha=1}^{r}g(B_{\alpha}X, Y)\xi_{\alpha}=0.
\end{equation}

Inverting $X$ by $Y$ in the last equality and summing these relations we obtain:
\[
\sum_{\alpha=1}^{r}g(Y, \xi _{\alpha})B_{\alpha}X+\sum_{\alpha=1}^{r}g(X, \xi _{\alpha})B_{\alpha}Y=0
\]
and multiplying this equality by $Z \in \mathcal{X}(M)$ we get:
\begin{equation} \label{e30}
\sum_{\alpha=1}^{r}g(Y, \xi _{\alpha})g(B_{\alpha}X,Z)+\sum_{\alpha=1}^{r}g(X, \xi _{\alpha})g(B_{\alpha}Y,Z)=0.
\end{equation}

Using (\ref{e28}) and (\ref{e30}) it follows that:
\begin{equation}\label{e31}
\sum_{\alpha=1}^{r}g(Y, \xi _{\alpha})g(B_{\alpha}X,Z)=0,
\end{equation}
for any $X,Y, Z \in \mathcal{X}(M)$.

If $M$ is non-invariant submanifold of codimension $r \geq 2$ in a locally metallic Riemannian manifold $(\overline{M},\overline{g}, J)$
and the normal connection $\nabla^{\bot}$ vanishes identically on the normal bundle $T^{\bot}(M)$,
from the last lemma we obtained  that the vector fields $\xi _{1},..., \xi _{r}$ are linearly independent in any point $x \in M$.
Due to these $r$ linearly independent vector fields on $M$ we have $r\leq n$. It follows that there exists a vector field $Y \in \mathcal{X}(M)$
which is orthogonal on the space spanned by $\{\xi _{1},..., \xi _{r}\}-\{\xi _{\alpha}\}$ and $g(Y, \xi _{\alpha})\neq 0$. Hence, we obtain $B_{\alpha}X=0$ which is equivalent to $PA_{\alpha}=A_{\alpha}P$, for any $\alpha \in \{1,..., r\}$.

For $r=1$ we have $g(Y, \xi)BX=0$, where $B=PA-AP$. For $Y=\xi $ we have $g(\xi , \xi )BX=0$.
But $g(\xi , \xi )=q+p a-a^{2} \neq 0$ (because $M$ is non-invariant submanifold) and we obtain $BX=0$, for any $X \in \mathcal{X}(M)$, which is equivalent to $PA=AP$.
\end{proof}

\begin{remark}
If $M$ is a non-invariant totally umbilical (or totally geodesic) $n$-dimensional submanifold of codimension $r \geq 1$
in a locally metallic Riemannian manifold $(\overline{M},\overline{g}, J)$ such that the normal connection $\nabla^{\bot}$ vanishes identically,
then the $\Sigma $-metallic Riemannian structure induced on the submanifold $M$ is normal.
\end{remark}


\linespread{1}
Cristina E. Hretcanu, \\ Stefan cel Mare University of Suceava, Romania, e-mail: criselenab@yahoo.com,\\
\linespread{1}
Adara M. Blaga, \\ West University of Timisoara, Romania, e-mail: adarablaga@yahoo.com.


\begin{thebibliography}{}

\bibitem{CrHr} M. Crasmareanu and C.E. Hretcanu, \textit{Golden differential geometry}, Chaos Solitons Fractals \textbf{38} (2008), no. 5, 1229--1238. MR2456523

\bibitem{CrHrMu} M. Crasmareanu, C.E. Hretcanu and M.I. Munteanu, \textit{Golden- and product-shaped hypersurfaces in real space forms}, Int. J. Geom. Methods Mod. Phys. 10 (2013), no. 4, 1320006, 9 pp. MR3037234

\bibitem{Goldberg1} S.I. Goldberg and K. Yano, \textit{Polynomial structures on manifolds}, Kodai Math. Sem. Rep. \textbf{22} (1970) 199--218. MR0267478

\bibitem{Goldberg2} S.I. Goldberg and N.C. Petridis, \textit{Differentiable solutions of algebraic equations on manifolds}, Kodai Math. Sem. Rep. \textbf{25} (1973), 111--128. MR0315627

\bibitem{Hr1} C.E. Hretcanu, \textit{Induced structures on spheres and product of spheres}, An. Stiint. Univ. Al. I. Cuza Iasi Mat. (N.S.), \textbf{54} (2008), no.1,  39--50

\bibitem{Hr2} C.E. Hretcanu and M. Crasmareanu, \textit{On some invariant submanifolds in a Riemannian manifold with Golden structure}, An. Stiint. Univ. Al. I. Cuza Iasi. Mat. (N.S.) \textbf{53} (2007), suppl. 1, 199--211. MR2522394

\bibitem{Hr3} C.E. Hretcanu and M.C. Crasmareanu, \textit{Applications of the Golden ratio on Riemannian manifolds}, Turkish J. Math. \textbf{33} (2009), no. 2, 179--191. MR2537561

\bibitem{Hr4} C.E. Hretcanu and M. Crasmareanu, \textit{Metallic structures on Riemannian manifolds}, Rev. Un. Mat. Argentina \textbf{54} (2013), no. 2, 15--27. MR3263648

\bibitem{Pitis} G. Pitis, \textit{On some submanifolds of a locally product manifold}, Kodai Math. J. \textbf{9} (1986), no. 3, 327--333. MR0856679

\bibitem{Spinadel} V.W. de Spinadel, \textit{The metallic means family and forbidden symmetries}, Int. Math. J. \textbf{2} (2002), no. 3, 279--288. MR1867157

\bibitem{Yano} K. Yano, M. Kon, \textit{Structures on Manifolds}, World Scientific, Singapore, Series in pure matematics, 3, (1984).

\end{thebibliography}
\end{document}